\documentclass[12pt,reqno]{amsart}

\usepackage{floatrow}

\usepackage{amsmath}
\usepackage{amssymb}
\usepackage[left=2.54cm,top=2.54cm,right=2.54cm,bottom=2.54cm]{geometry}
\usepackage{epsfig}
\usepackage[normalem]{ulem}
\usepackage[usenames,dvipsnames]{color}
\usepackage{todonotes}
\usepackage[colorlinks, citecolor=black, linkcolor=black, urlcolor=black]{hyperref}
\usepackage{bm}
\usepackage{algorithm}
\usepackage{algpseudocode}

\usepackage{enumitem}
\newcommand{\subscript}[2]{$#1 _ #2$}

\usepackage{hyperref}

\usepackage{tikz}
\usetikzlibrary{backgrounds}
\usetikzlibrary{intersections}

\newtheorem{theorem}{Theorem}[section]
\newtheorem{lemma}[theorem]{Lemma}
\newtheorem{proposition}[theorem]{Proposition}

\newtheorem{definition}[theorem]{Definition}

\newcommand\eps{\varepsilon}
\newcommand{\E}{\mathbb E}

\newcommand{\Nn}{{\mathbb N}}

\newcommand{\scr}{\mathcal}
\newcommand{\mb}{\mathbb}
\newcommand{\til}{\widetilde}

\newcommand{\ex}{\mathrm{ex}}

\def\R{{\mathcal R}}
\def\ex{{\mathbb E}}

\title{Perfect Matchings in the Semi-random Graph Process}
\author{Pu Gao}
\address{Department of Combinatorics and Optimization, University of Waterloo, Waterloo, Canada}
\email{pu.gao@uwaterloo.ca}

\author{Calum MacRury}
\address{Department of Computer Science, University of Toronto, Toronto, Canada}
\email{cmacrury@cs.toronto.edu}

\author{Pawe\l{} Pra\l{}at}
\address{Department of Mathematics, Ryerson University, Toronto, Canada}
\email{pralat@ryerson.ca}

\date{}

\begin{document}

\maketitle

\begin{abstract}
The semi-random graph process is a single player game in which the player is initially presented an empty graph on $n$ vertices. In each round, a vertex $u$ is presented to the player independently and uniformly at random. The player then adaptively selects a vertex $v$, and adds the edge $uv$ to the graph. For a fixed monotone graph property, the objective of the player is to force the graph to satisfy this property with high probability in as few rounds as possible.

We focus on the problem of constructing a perfect matching in as few rounds as possible. In particular, we present an adaptive strategy for the player which achieves a perfect matching in $\beta n$ rounds, where the value of $\beta < 1.206$ is derived from a solution to some system of differential equations. This improves upon the previously best known upper bound of $(1+2/e+o(1)) \, n < 1.736 \, n$ rounds. We also improve the previously best lower bound of $(\ln 2 + o(1)) \, n > 0.693 \, n$ and show that the player cannot achieve the desired property in less than $\alpha n$ rounds, where the value of $\alpha > 0.932$ is derived from a solution to another system of differential equations. As a result, the gap between the upper and lower bounds is decreased roughly four times.
\end{abstract}

\section{Introduction and Main Results}

\subsection{Definitions} 

In this paper, we consider the \textbf{semi-random process} suggested by Peleg Michaeli and studied recently in \cite{beneliezer2019semirandom,beneliezer2020fast,gao2020hamilton} that can be viewed as a ``one player game''. The process starts from $G_0$, the empty graph on the vertex set $[n]:=\{1,\ldots,n\}$ where $n \ge 1$. In each step $t$, a vertex $u_t$ is chosen uniformly at random from $[n]$. Then, the player (who is aware of graph $G_t$ and vertex $u_t$) must select a vertex $v_t$ and add the edge $u_tv_t$ to $G_t$ to form $G_{t+1}$. The goal of the player is to build a (multi)graph satisfying a given property $\scr{P}$ as quickly as possible. It is convenient to refer to $u_t$ as a {\bf square}, and $v_t$ as a {\bf circle} so every edge in $G_t$ joins a square with a circle. We say that vertex $j \in [n]$ is \textbf{covered} by the square $u_t$ arriving at round $t$,
provided $u_t = j$. The analogous definition extends to the circle $v_t$. Equivalently, we may view $G_t$ as a directed graph where each arc  directs from $u_t$ to $v_t$. For this paper, it is easier to consider squares and circles for counting arguments.

A \textbf{strategy} $\scr{S}$ is defined by specifying for each $n \ge 1$, a sequence of functions $(f_{t})_{t=1}^{\infty}$, where for each $t \in \mb{N}$, $f_t(u_1,v_1,\ldots, u_{t-1},v_{t-1},u_t)$ is a distribution over $[n]$
which depends on the vertex $u_t$, and the history of the process up until step $t-1$. Then, $v_t$ is chosen according to this distribution. If $f_t$ is an atomic distribution, then $v_t$ is determined by $u_1,v_1, \ldots ,u_{t-1},v_{t-1},u_t$. Observe that this means that the player needs to select her strategy in advance, before the game actually starts. We then denote $(G_{t}^{\scr{S}}(n))_{i=0}^{t}$ as
the sequence of random (multi)graphs obtained by following the strategy $\scr{S}$ for $t$ rounds; where we shorten $G_{t}^{\scr{S}}(n)$
to $G_t$ or $G_{t}(n)$ when clear.

Suppose $\scr{P}$ is a monotonely increasing property. Given a strategy $\scr{S}$ and a constant $0<q<1$, let $\tau_{\scr{P}}(\scr{S},q,n)$ be the minimum $t \ge 0$ for which $\mb{P}[G_{t} \in \scr{P}] \ge q$,
where $\tau_{\scr{P}}(\scr{S},q,n):= \infty$ if no such $t$ exists. Define
\[
\tau_{\scr{P}}(q,n) = \inf_{ \scr{S}} \tau_{\scr{P}}( \scr{S},q,n),
\]
where the infimum is over all strategies on $[n]$. 
Observe that for each $n \ge 1$, if $0 \le q_{1} \le q_{2} \le 1$, then $\tau_{\scr{P}}(q_1,n) \le \tau_{\scr{P}}(q_2,n) $ as $\scr{P}$ is increasing. Thus, the function $q\rightarrow \limsup_{n\to\infty} \tau_{\scr{P}}(q,n)$ is non-decreasing,
and so the limit,
\[
\tau_{\scr{P}}:=\lim_{q\to 1^-}\limsup_{n\to\infty} \frac{\tau_{\scr{P}}(q,n) }{n},
\]
is guaranteed to exist. The goal is typically to compute upper and lower bounds on $\tau_{\scr{P}}$
for various properties $\scr{P}$.

\subsection{Main Results}

In this work, we focus on the property of having a perfect matching, which we denote by ${\tt PM}$. We remark that by convention, we say that a graph on an odd number of vertices has a perfect matching, if the matching saturates all but one vertex. Hence, if $n$ is odd, then $\tau_{{\tt PM}}(q,n)\le \tau_{{\tt PM}}(q,n+1)$ by the following natural coupling between the two corresponding processes. Indeed, the first strategy on $n$ vertices may steal the second strategy on $n+1$ vertices if the square and the circle are both on $[n]$. Otherwise, the square and the circle are placed uniformly at random on $[n]$. Immediately the subgraph induced by $[n]$ by the second strategy is a subgraph of the graph constructed by the first strategy and the inequality follows. Since $\tau_{{\tt PM}}$ is an asymptotic definition in $n$, it suffices to consider only even $n$. Note also that since we focus on creating perfect matchings, we shall hereby restrict our attention to strategies which do not create self-loops.

Our first result is an improvement on the upper bound of $\tau_{\texttt{PM}}$ from $1+2/e < 1.73576$ to $1.20524$. The current upper bound of $1+2/e$ follows from the fact, first observed in~\cite{beneliezer2019semirandom}, that one may couple the semi-random process with the following process that generates a random bipartite graph. This process is known to have a perfect matching with probability tending to $1$ as $n \rightarrow \infty$ (\emph{a.a.s.})~\cite{pittel}. The graph has two bipartite parts of size $n/2$ and is generated in two rounds. The first round is the bipartite version of the 1-out process. That is, each vertex chooses an out-neighbour independently and uniformly at random from the other vertex part. A vertex is classified as unpopular if it has been chosen by at most one vertex. In the second round, each unpopular vertex chooses another out-neighbour independently and uniformly at random from the other vertex part. The final bipartite graph is obtained by ignoring the directions of the arcs. 

We propose a fully adaptive algorithm to construct a semi-random graph with a perfect matching, which gives the following upper bound on $\tau_{\texttt{PM}}$.

\begin{theorem} \label{thm:main_upper_bound}
$\tau_{\texttt{PM}} \le \beta + 10^{-5} \le 1.20524$, where $\beta$ is derived
from a system of differential equations.
\end{theorem}

Our second result is an improvement on the lower bound of $\tau_{\texttt{PM}}$ from $\ln(2) \ge 0.69314$ to $0.93261$.
First observe that trivially $\tau_{{\tt PM}} \ge 1/2$, as any strategy of the player must wait at least $n/2$ rounds in order to build a perfect matching. In fact, this lower bound can be improved to $\ln(2) \ge 0.69314$, as was first observed in~\cite{beneliezer2019semirandom}. There are two obvious necessary conditions for the existence of a perfect matching, both giving exactly the same lower bound. If there exists a perfect matching in $G_t$, then $G_t$ has the minimum degree at least 1 and there are at least $n/2$ vertices in $G_t$ with at least one square. In order to warm-up we re-prove this result in Subsection~\ref{sec:warm-up}.

\begin{proposition}[\cite{beneliezer2019semirandom}] \label{prop:warm_up_lowerbound}
$\tau_{\texttt{PM}} \ge \ln(2) \ge 0.69314$.
\end{proposition}
The precise statement of our lower bound is in terms of a root of a function. Specifically,
define
\[
\alpha = \inf\{b\ge 0: g(b) \ge 1/2 \},
\]
where
\begin{equation*}
    g(b) := 1 + \frac {1-2b}{2} \exp(-b) - (b+1) \exp(-2b) - \frac {1}{2} \exp(-3b) 
\end{equation*}
We prove the following:
\begin{theorem} \label{thm:main_lower_bound}
$\tau_{\texttt{PM}} \ge \alpha \ge 0.93261$.
\end{theorem}


\subsection{Previous Results} 

Let us briefly describe a few known results on the semi-random process. In the very first paper~\cite{beneliezer2019semirandom}, it was shown that the process is general enough to approximate (using suitable strategies) several well-studied random graph models. In the same paper, the process was studied for various natural properties such having minimum degree $k \in \Nn$ or having a fixed graph $H$ as a subgraph. In particular, it was proven that \emph{a.a.s.}\ one can construct $H$ in less than $\omega \, n^{(d-1)/d}$ rounds where $d \ge 2$ is the degeneracy of $H$, and $\omega=\omega(n)$ is any function tending to infinity as $n \to \infty$. It was conjectured that their general lower bound is sharp. This conjecture was recently proven in~\cite{behague2021subgraph}.

The property of having a Hamilton cycle, which we denote by ${\tt HAM}$, was also studied for the semi-random process. As observed in~\cite{beneliezer2019semirandom}, if $G_t$ has a Hamilton cycle, then $G_t$ has the minimum degree at least 2 yielding $\tau_{\tt HAM} \ge \ln 2+ \ln(1+\ln2) \ge 1.21973$. On the other hand, it is known that the famous $3$-out process is \emph{a.a.s.}\ Hamiltonian~\cite{3out}. As the semi-random process can be coupled with the $3$-out process, we get that $\tau_{\tt HAM} \le 3$. A new upper bound was obtained in~\cite{gao2020hamilton} in terms of an optimal solution of an optimization problem, whose value is believed to be $2.61135$ by numerical support. In the same paper, the lower bound mentioned above was shown
to not be tight. Recently, the authors of this paper managed to analyze a fully adaptive algorithm further improving the upper bound and managed to substantially improve the lower bound. These results will be included in the forthcoming paper~\cite{draft}.

Finally, let us mention about the property of containing a given spanning graph $H$ as a subgraph. It was asked by Noga Alon whether for any bounded-degree $H$, one can construct a copy of $H$ \emph{a.a.s.}\ in $O(n)$ rounds.  This question was answered positively in a strong sense in~\cite{beneliezer2020fast}, in which is was shown that any graph with maximum degree $\Delta$ can be constructed \emph{a.a.s.}\ in $(3\Delta/2+o(\Delta))n$ rounds. They also proved that if $\Delta = \omega (\log(n))$, then this upper bound improves to $(\Delta/2 + o(\Delta))n$ rounds. Note that both of these upper bounds are asymptotic in $\Delta$. When $\Delta$ is constant in $n$, such as in both the perfect matching and Hamiltonian cycle setting, determining the optimal dependence on $\Delta$ for the number of rounds needed to construct $H$ remains open.

\section{Upper Bound}

In this section, we describe an adaptive algorithm which builds a perfect matching in around $1.2n$ rounds \emph{a.a.s.}
In order to warm-up, we start with a slightly simpler algorithm to analyze and then discuss the adjustments needed to claim the desired upper bound.
Note that throughout the paper, floor is automatically applied to any number that needs to be an integer (typically a function of $n$).

\subsection{Warming-up}\label{warm-up-upper-bound} 
In this section we analyse a simplified algorithm which constructs a semi-random graph with a perfect matching in less than $1.28n$ steps.
The algorithm operates in two stages. 
In the first stage, the algorithm keeps building the matching greedily whenever possible, but will keep $u_tv_t$ for future augmentation if $u_t$ lands on a saturated vertex.
In each step of the algorithm we will keep track of a matching $M_t$ that is currently built and the set $U_t$ of unsaturated vertices. We will also maintain a set of red vertices which is a subset of vertices that are saturated by $M_t$. 
If a red vertex is hit by a square, then an augmenting path will be created and we can update $M_t$ to a larger matching.
In order to briefly explain the idea, imagine $uv$ is an edge in the matching. If a square lands on $u$ then we will place its corresponding circle on an unsaturated vertex $x$, and then colour $v$ red. If in the future a square is placed on $v$, then we will place its corresponding circle on an unsaturated vertex $y$. Now the path $xuvy$ is an augmenting path. When the red vertex $v$ receives a square, the algorithm needs to know the neighbour of $u$ that is unsaturated, and there might be several of them. In this simple algorithm, we keep track of only the first square that landed on $u$. 

We now formally describe the first stage of the algorithm.
Suppose that at step $t$, $u_t$ lands on a saturated vertex $u$ and $v_t$ is the corresponding circle which is necessarily placed on an unsaturated vertex. Colour vertex $u$ as well as the edge $u_tv_t$ green, and colour the mate of $u$ in the matching red.
Below is a formal description of the $t$-th step of the algorithm.
\begin{enumerate}[label=(\subscript{A}{{\arabic*}})]
    \item If $u_t\in U_{t-1}$ then let $v_t$ be a uniformly random vertex in $U_{t-1}$. Let $M_t=M_{t-1}\cup \{ u_tv_t \}$ and $U_t=U_{t-1}\setminus \{u_t,v_t\}$. For every green vertex $x$, if it is adjacent to either $u_t$ or $v_t$ by a green edge then uncolour this green edge and uncolour $x$ (from green) and uncolour the mate of $x$ in $M_{t}$ (from red). \label{eqn:randomized_case_1}
    \item If $u_t$ lands on a red vertex, then let $v_t$ be a uniformly random vertex in  $U_{t-1}$. Let $x$ be the mate of $u_t$ in $M_{t-1}$. Let $y$ be the vertex in $U_{t-1}$ which is adjacent to $x$ by a green edge. Let $M_t$ be the matching obtained by augmenting along the path $yxu_tv_t$, and let $U_t = U_{t-1} \setminus \{y, v_t\}$. Update the green vertices and edges and the red vertices accordingly as in \ref{eqn:randomized_case_1}. \label{eqn:randomized_case_2}
    \item If $u_t$ lands on an uncoloured saturated vertex, then choose $v_t$ uniformly at random from $U_{t-1}$. Colour the edge $u_tv_t$ and the vertex where $u_t$ lands on green and colour the mate of $u_t$ in $M_{t-1}$ red. Let $M_t=M_{t-1}$ and $U_t = U_{t-1}$. \label{eqn:randomized_case_3}
    \item If $u_t$ lands on a green vertex, then let $v_t$ be an arbitrary vertex in $[n]$. Let $M_t=M_{t-1}$ and $U_t = U_{t-1}$. The edge $u_tv_t$ will not be used in the matching construction in the future.
\end{enumerate}
The first stage terminates at the step when $|U_t|$ becomes at most $\eps n$ where $\eps=10^{-14}$. In the second stage, we design the algorithm which saturates the remaining unsaturated vertices in at most $100\sqrt{\eps}n=10^{-5}n$ steps. We will define and analyse the second stage after analysing the first.

\medskip
Let $X(t)=2|M_t|$ and let $R(t)$ denote the number of red vertices. By the algorithm, $R(t)$ is also equal to the number of green vertices, and thus equal to the number of green edges.
Let $H_t:=(X(i), R(i))_{0\le i\le t}$. Note that $H_t$ does \textit{not} encompass the full history of random graph process at time $t$ (i.e., $G_0, \ldots ,G_t$ -- the first $t+1$
graphs constructed by the randomized algorithm).
We condition on less information so that the circle placements amongst the unsaturated vertices remain distributed uniformly at random. This allows
us to claim the following expected difference equations:
\begin{eqnarray*}
\E[X(t+1)-X(t)\mid H_t]&=&\frac{2(n-X(t) + R(t))}{n} + O(1/n),\\
\E[R(t+1)-R(t)\mid H_t]&=& \frac{n-X(t)}{n} \cdot \frac{-2R(t)}{n-X(t)}+ \frac{R(t)}{n}    \left(-1- \frac{2(R(t)-1)}{n-X(t)} \right) \\
&& + \frac{X(t)-2 R(t)}{n} + O(1/n).
\end{eqnarray*}
The first equation above is obvious, as the probability that $u_{t+1}$ lands on a vertex in $U_t$ or a red vertex is $(n-X(t)+R(t))/n$. The probability that $v_{t+1}$ is on the same vertex as $u_{t+1}$ in \ref{eqn:randomized_case_1}, or is on vertex $y$ in \ref{eqn:randomized_case_2}, is $O(1/n)$.
In order to justify the second equation, note that if $u_{t+1}$ lands on a vertex in $U_t$, then two vertices $z_1$ and $z_2$ in $U_t$ become saturated after the augmentation. 
Since the ends of the set of green edges in $U_t$ are uniformly distributed,  the expected number of green edges incident with $z_1$ or $z_2$ is equal to $2R(t)/(n-X(t))$.
The other ends of these green edges become uncoloured from green after the augmentation which, in turn, forces their mates to become uncoloured from red. If $u_{t+1}$ lands on a red vertex, the situation is similar, except that the mate of $u_{t+1}$ is first uncoloured from green. If $u_{t+1}$ lands on a green vertex, then there is no change to $R(t)$. Finally, if $u_{t+1}$ lands on a vertex that is neither unsaturated nor coloured, then a new green vertex is created. This explains the second equation above.

By writing $x(s)=X(sn)/n$ and $r(s)=R(sn)/n$, we have that
\begin{eqnarray}
x'&=&2(1-x+z),\label{x'}\\
r'&=&\frac{-2r}{1-x}(1-x+r)-r+x-2r, \label{y'}
\end{eqnarray}
with the initial conditions $x(0)=r(0)=0$. 

The right hand sides of~(\ref{x'}) and~(\ref{y'}) are continuous and Lipschitz in the connected open set $D=\{(s,x,r): -\eps<s<3, |r|<2, -\eps<x<1-\eps/4\}$ which contains the point $(0,x(0),z(0))=(0,0,0)$. Let $T_D$ be the first step $t$ where $(t/n,X(t)/n, R(t)/n)\notin D$. Then,  $|X(t+1)-X(t)|\le 1$ for every $t<T_D$, and with probability $O(n^{-2})$, $|R(t+1)-R(t)|\le \log^2 n$ for every $t<T_D$ (in order to see the second assertion, note that, in expectation, every unsaturated vertex is adjacent to $R(t)/|U_t|\le 4/\eps$ green vertices, and thus the claim easily follows by standard bounds on the tail probability of a binomial variable). By the differential equation method of Wormald (see, for example,~\cite[Theorem 5.1]{de} or~\cite{warnke2019wormalds}) (applied with $\gamma=O(n^{-2})$, $\beta=\log^2 n$ and $\lambda=n^{-1/4}$),
the differential equations~(\ref{x'}) and~(\ref{y'}) with the given initial conditions have a unique solution that can extend arbitrarily close to the boundary of $D$, and \emph{a.a.s.}\ $X(t)=x(t/n)n+O(\lambda n)$ for every $t$ where $t/n<\sigma$ where $\sigma$ is the supremum of $s$ where $x(s)\le 1-\eps/3$ and $s<3$.
Numerical calculations show that $x$ reaches $1-\eps/3$ before $s=1.2769497$.\footnote{Maple worksheet may be accessed at \texttt{https://math.ryerson.ca/\~{}pralat/}.} Hence, \emph{a.a.s.}\ after $1.2769497n$ steps there are at most $\eps n$ unsaturated vertices. 

\medskip

Finally, we describe the second stage of the algorithm, which we refer to as the \textbf{clean-up algorithm}.  Let $j_0$ be the number of unsaturated vertices after the first stage. Thus, $j_0\le \eps n$. Let $j_k=\lfloor(3/4)^k j_0\rfloor$, for $k=1,2,\ldots$. 
For each $k\ge 1$, the algorithm saturates $j_{k-1}-j_k$ unsaturated vertices by doing the following.

\begin{enumerate}
    \item[(i)] Uncolour all vertices and squares in the whole graph.
    \item[(ii)] Sprinkle $\lfloor\sqrt{3 j_{k-1}n}/4\rfloor$ semi-random edges $u_tv_t$ as follows. If $u_t$ lands on a saturated vertex $x$, then let $v_t$ be an unsaturated vertex selected uniformly at random, and colour the mate of $x$ in $M_t$ by red (which is possibly already red). If $u_t$ lands on an unsaturated vertex, then let $v_t$ be an arbitrary vertex, and the algorithm will not use this edge for constructing the matching. 
    \item[(iii)] Let $\R$ be the set of red vertices after (ii). At each step $t$, if $u_t$ lands on a red vertex or an unsaturated vertex then let $v_t$ be a uniformly random unsaturated vertex, and then augment $M_t$ as in \ref{eqn:randomized_case_1} and \ref{eqn:randomized_case_2} as in the first stage, and update $\R$ accordingly. Stop when $|U_t|=j_k$. Let $T_k$ be the total number of steps in (ii) and (iii). 
\end{enumerate}

The total number of steps in the final stage of constructing a perfect matching is then $\sum_{k\ge 1} T_k$. Next we bound $T_k$ for every $k$. Of course, the number of steps in (ii) is at most $\sqrt{3 j_{k-1}n}/4$. It is easy to show by standard concentration arguments that with probability $1-O(n^{-2})$, $|\R| \ge \sqrt{3 j_{k-1}n}/5$ after (ii). (Indeed, since $j_{k-1} \le \eps n$ and the number of rounds is $\sqrt{3 j_{k-1}n}/4 = o(n)$, $\E[|\R|] \ge (1-\eps+o(1)) \sqrt{3 j_{k-1}n}/4$.) During the execution of (iii), for each augmentation that is performed, the expected number of red vertices that become uncoloured is at most 
\[
1+2\cdot \frac{\sqrt{3 j_{k-1}n}/4}{j_k},
\]
since the total number of red vertices is at most $\sqrt{3 j_{k-1}n}/4$, and the total number of unsaturated vertices in each round is at least $j_k$.
There are $\frac{1}{2}(j_{k-1}-j_k)$ total augmentations in (iii). Thus, by standard concentration arguments, with probability $1-O(n^{-2})$, at the end of (iii), 
\[
|\R|\ge \frac{\sqrt{3 j_{k-1}n}}{5}-2\cdot\left(1+ 2\cdot \frac{\sqrt{3 j_{k-1}n}/4}{j_k} \right) \cdot \frac{j_{k-1}-j_k}{2} = \frac{\sqrt{3j_{k-1}n}}{30}-\frac{j_{k-1}}{4}\ge \frac{\sqrt{2j_{k-1}n}}{30},
\]
where the last inequality holds since $j_{k-1}\le \eps n$ where $\eps=10^{-14}$ by our choice.
Consequently, the probability that a square lands on a red vertex at each step in (iii) is at least $\sqrt{2j_{k-1}}/30\sqrt{n}$. 
Hence, 
\[
\ex T_k\le \frac{1}{4}\sqrt{3j_{k-1}n}+\left(\frac{\sqrt{2j_{k-1}}}{30\sqrt{n}}\right)^{-1}\frac{j_{k-1}-j_k}{2}<3.1\sqrt{j_{k-1}n}.
\]
Then, again by standard concentration arguments, \emph{a.a.s.}\ 
\[
T_k\le \max\{4\sqrt{j_{k-1}n},\log^2 n\} =4\sqrt{j_{k-1}n} \le  4\sqrt{(3/4)^{k-1}\eps}\cdot n,\ \  \mbox{for every $k\ge 1$}.
\]
Thus, \emph{a.a.s.}\ the total number of steps consumed in the second stage is at most
\[
\sum_{k\ge 1} T_k \le \sum_{k\ge 1} 4\sqrt{(3/4)^{k-1}\eps}\cdot n <100\sqrt{\eps} n.
\]
By our choice of $\eps$, $100\sqrt{\eps}=10^{-5}$, \emph{a.a.s.}\ the total number of steps the algorithm spend on constructing a perfect matching is at most
$(1.2769497+10^{-5})n<1.27696n$.

\subsection{Improving the Upper Bound}

In this section, we describe a modified algorithm which allows us to get the claimed upper bound
of Theorem \ref{thm:main_upper_bound}. As before, the algorithm keeps building the matching greedily whenever possible. However, instead of
making random choices, it will choose its circles both deterministically and in a strategic way. Specifically,
suppose that $u_t$ lands on a saturated vertex. In this case, we choose $v_t$ such that $v_t$ has the minimum number of circles
on it amongst the vertices of $U_{t-1}$. We make the analogous decision when $u_t$ lands on an unsaturated vertex,
as well as when $u_t$ lands on a red vertex and we wish to extend the matching via an augmenting path. We now formally describe our algorithm
for round $t$:
\begin{enumerate}[label=(\subscript{B}{{\arabic*}})]
    \item If $u_t\in U_{t-1}$, then let $v_t$ be a vertex of $U_{t-1}$ with a minimum number of circles among those in $U_{t-1}\setminus\{u_t\}$. Let $M_t=M_{t-1}\cup \{ u_tv_t \}$ and $U_t=U_{t-1}\setminus \{u_t,v_t\}$. For every green vertex $x$, if it is adjacent to either $u_t$ or $v_t$ by a green edge then uncolour this green edge and $x$ (from green) and uncolour the mate of $x$ in $M_{t}$ (from red). \label{eqn:deterministic_case_1}
    \item If $u_t$ lands on a red vertex, then let $x$ be the mate of $u_t$ in $M_{t-1}$. Let $y$ be the vertex in $U_t$ which is adjacent to $x$ by a green edge. Let $v_t$ be a vertex of $U_{t-1}$ with a minimum number of circles among those in $U_{t-1}\setminus\{y\}$.   Let $M_t$ be the matching obtained by augmenting along the path $yxu_tv_t$, and let $U_t = U_{t-1} \setminus \{y, v_t\}$. Update the green vertices and edges and the red vertices accordingly as in \ref{eqn:deterministic_case_1}.  \label{eqn:deterministic_case_2}
    \item If $u_t$ lands on an uncoloured saturated vertex, then let $v_t$ be a vertex of $U_{t-1}$ with a minimum number of circles. Colour the vertex where $u_t$ lands on green and colour the mate of $u_t$ in $M_{t-1}$ red. Colour the edge $u_tv_t$ green. Let $M_t=M_{t-1}$ and $U_t = U_{t-1}$.   \label{eqn:deterministic_case_3}
    \item If $u_t$ lands on a green vertex, then let $v_t$ be an arbitrary vertex in $[n] \setminus U_{t-1}$. (We restrict to $[n] \setminus U_{t-1}$ so that we can conveniently keep track of the number of circles on our unsaturated vertices.) Let $M_t=M_{t-1}$ and $U_t = U_{t-1}$. The edge $u_tv_t$ will not be used in the matching construction in the future.  \label{eqn:deterministic_case_4}
\end{enumerate}
As before, we define $X(t)$ to be the number of saturated vertices after $t \ge 0$ rounds. 
Given $q \ge 0$, we define $\til{R}_{q}(t)$ to be the number of unsaturated vertices with
precisely $q$ circles after $t$ rounds. Note that $\til{R}_{q}(0)=0$ for all $q \ge 1$, and 
$\til{R}_{0}(0)=n$. Now, for $q \ge 1$, we say that a red vertex is of \textbf{type} 
$q$ at round $t$, provided its mate in $M_{t}$ is
adjacent to an unsaturated vertex via a green edge with precisely $q$ circles on it. 
We denote $R_{q}(t)$ as the number of red vertices of type $q$ at time $t \ge 0$.
Observe that $R_{q}(t) = q \til{R}_{q}(t)$ for $q \ge 1$ and $t \ge 0$. Observe also that for every $t\ge 0$, there is $q$ where all red vertices are either of type $q$ or of type $q+1$.

For each $q \ge 1$, let us define the stopping time $\tau_{q}$ to
be the smallest $t \ge 0$ such that $\til{R}_{j}(t) = 0$ for all $j<q$. It is obvious that $\tau_q$ is well-defined and is non-decreasing in $q$. Moreover, after the step $\tau_q$, all unsaturated vertices in the graph have precisely $q$ circles. This is obvious if what happened in step $\tau_q$ is in case \ref{eqn:deterministic_case_3}. Case \ref{eqn:deterministic_case_4} cannot happen in step $\tau_q$, as the number of circles on the unsaturated vertices does not change in this case. If case \ref{eqn:deterministic_case_1} or \ref{eqn:deterministic_case_2} occurs in step $\tau_q$, it is possible that $v_t$ receives more than $q$ circles in some rare situations, but $v_t$ becomes saturated after the augmentation and so the number of circles on every other unsaturated vertex is equal to $q$.
It follows then that  $\til{R}_{q}(t) = |U_t| = n - X(t)$.  By definition, $\tau_{0}=0$. Let us refer to \textbf{phase} $q$ as those $\tau_{q-1} \le t < \tau_{q}$.


We shall analyze our algorithm for the first $k$ phases using the differential equation
method. For the case when $k = 1100$, we show that \emph{a.a.s.}\ $X(\tau_k) \ge (1-10^{-6}) n$ and $\tau_{k} \le 1.20365 n $. We complete the perfect matching
by running the randomized algorithm from Subsection~\ref{warm-up-upper-bound}.\footnote{Maple worksheet may be accessed at \texttt{https://math.ryerson.ca/\~{}pralat/}.} The only difference is that the initial conditions are $x(0)=1-10^{-6}$ (instead of $x(0)=0$) and $z(0)=0$. The conclusion is that \emph{a.a.s.}\ after additional $0.00158n$ steps there are at most $10^{-14}n$ unsaturated vertices. As before, the clean-up algorithm for the remaining unsaturated vertices takes at most $0.00001n$ rounds, yielding the upper bound of $(1.20365 + 0.00158 + 0.00001)n = 1.20524n$.

\subsubsection{Analyzing the $k$ phases}

Unlike in the analysis of the randomized algorithm, we can safely condition on
the full history of the random graph process in each phase. 
If $G_{0},\ldots ,G_{t}$ correspond to the first $t+1$ graphs constructed
by the deterministic algorithm, then let $H_{t}:=(G_0,\ldots, G_t)$.

In phase $1$, we place circles on the unsaturated vertices, thus creating red vertices of type~$1$. This leads to the following
equations regarding the expected changes of $X(t)$ and $R_1(t)$, for $0 \le t < \tau_1$: 
\begin{eqnarray*}
\E[X(t+1)-X(t) \mid H_t]&=& 2 \left(\frac{n-X(t) + R_{1}(t)}{n} \right),\\
\E[R_{1}(t+1)- R_{1}(t) \mid H_t]&=& \left( \frac{X(t) - 2R_{1}(t)}{n}\right) - \frac{2 R_{1}(t)}{n} + O(1/n).
\end{eqnarray*}
Observe that the first equation follows, since $2$ vertices become saturated whenever
a square lands on a red vertex, or an unsaturated vertex. The second equation follows
since we create an additional red vertex whenever a square lands on an uncoloured saturated
vertex (of which there exist $X(t) - 2R_{1}(t)$). We destroy a red vertex whenever a square
lands on a type $1$ red vertex, or an unsaturated vertex with one circle on it. There
are $R_{1}(t)$ of each of these. Note that the $O(1/n)$ term accounts for contributions from the following two cases. In the first case, $u_t$ lands on an unsaturated vertex which is the unique one in $U_t$ which is not covered by a circle. In that case, $R_1(t)$ decreases by one, but the probability that this case occurs is $O(1/n)$. In the other case, $u_t$ lands on a red vertex as in case \ref{eqn:deterministic_case_2} where $y$ happens to be the only unsaturated vertex in $U_t$ that is not covered by any circle. Again, it is easy to see that the expected contribution from this case is $O(1/n)$. 

Suppose that the functions $x_1,y_1: [0, \infty) \rightarrow \mb{R}$
are the the unique solution to the following system of differential equations:
\begin{eqnarray*} 
    x_1' &=& 2(1 - x_1 + y_1) \\
    y_{1}' &=& x - 4y_1,
\end{eqnarray*}
with initial conditions $x_1(0)=y_{1}(0)=0$.
Let  $c_1 > 0$ be the unique value such that $y_{1}(c_1)= 1 - x_{1}(c_1)$.
By applying the differential equation method, we can derive the following:
\begin{proposition} \label{prop:phase_one}
There exists a function $\eps_1(n) =o(1)$, such
that $a.a.s$ for all $0 \le t \le \tau_1$ it holds that
\begin{equation*}
    \max\{|X(t) -x_1(t/n) n|,|R_{1}(t) - y_{1}(t/n) n|\} \le \eps_1 n.
\end{equation*}
Moreover, $|\tau_{1}-c_{1} n| \le \eps_1 n$ $a.a.s$.
\end{proposition}
Let us now consider phase $q \ge 2$.
Observe that there exist red vertices only of types $q-1$ and $q$ during phase $q$. 
Thus, we can track the total number of red vertices at time $t$ by just focusing on 
the random variables $R_{q-1}(t)$ and $R_{q}(t)$.
Specifically, observe the following for $\tau_{q-1} \le t < \tau_{q}$:
\begin{eqnarray*}\label{eqn:phase_expected_changes}
\E[X(t+1)-X(t) \mid H_t]&=&2 \left(\frac{n-X(t)}{n}+\frac{R_{q-1}(t)}{n} + \frac{R_{q}(t)}{n} \right),\\
\E[R_{q}(t+1)- R_{q}(t) \mid H_t]&=& q\left( \frac{(X(t) - 2R_{q}(t) - 2 R_{q-1}(t)) - R_{q}(t)}{n}\right) \\
&&- \frac{R_{q}(t)}{n} + O(1/n), 
\end{eqnarray*}
where
\[
R_{q-1}(t)=(q-1)\left(n-X(t)-\frac{R_q(t)}{q}\right).
\]
The first equation follows since $2$ vertices become saturated whenever a square lands on either an unsaturated
vertex, or a red vertex. The second equation describes the expected change in the red vertices of type
$q$. Note that $q$ such vertices are created when a square lands on a saturated and uncoloured
vertex, of which there are precisely $X(t) - 2R_{q}(t) -2R_{q-1}(t)$. Similarly, $q$ red vertices of type $q$ are destroyed
when the square lands on a red vertex of type $q$, as well as when the square lands on an unsaturated vertex with precisely
$q$ circles. Observe that there are $R_{q}(t)$ of the former, and $\til{R}_{q}(t)=  R_{q}(t)/q$  of the latter.
Finally, the $O(1/n)$ term accounts for the rare case when $v_t$ is forced to be placed on a vertex with $q$ circles for the same reason as in phase 1.



We define
a system of differential equations for each phase $q \ge 2$, whose initial conditions
are determined inductively. The differential equation method implies that the trajectories of the random variables $X(t)$ and $R_q(t)$ are concentrated around  the solutions of these differential equations in every phase. Let us assume that
we have defined a system for phase $q-1$, whose solution has functions $x_{q-1}$ and $y_{q-1}$, and a value $c_{q-1} \ge 0$
such that $y_{q-1}(c_{q-1}) = 1 - x_{q-1}(c_{q-1})$  (observe that the base case holds by Proposition~\ref{prop:phase_one}). Note that we may interpret $c_{q-1} n$ as approximately the point in
the process at which all the unsaturated vertices have $q-1$ circles, thus corresponding to the end of phase $q-1$.
Define $x_{q}, y_{q}, z_{q}:[c_{q-1}, \infty) \rightarrow \mb{R}$ 
to be the unique solution to the following system of differential equations:
\begin{eqnarray*}
    x_{q}' &=& 2(1 - x_q + y_q + z_{q}), \\
    y_{q}' &=& q (x_q - 3 y_{q} - 2z_{q}) - y_{q},
\end{eqnarray*}
with initial conditions $x_{q}(c_{q-1})=x_{q-1}(c_{q-1})$, $z_{q}(c_{q-1}) = y_{q-1}(c_{q-1})$, and $y_{q}(c_{q-1})=0$, where $z_q=(q-1)(1-x-y_q/q)$.
Let $c_{q} \ge 0$ be the unique value such that $y_{q}(c_q) = 1- x_{q}(c_q)$ -- observe
that this is precisely the value such that $z_{q}(c_q)=0$, 
for all $q \ge 2$.
By applying the differential equation method inductively for each phase $q$, we get the following:
\begin{proposition} \label{prop:phase_q}
For each $q \ge 1$, there exists a function $\eps_q(n) =o(1)$, such
that $a.a.s$, for all $0 \le t \le \tau_q$ it holds that
\begin{equation*}
    \max\{|X(t) -x_q(t/n) n|,|R_{q}(t) - y_{q}(t/n) n|\} \le \eps_q n.
\end{equation*}
Moreover, $|\tau_{q}-c_{q} n| \le \eps_q n$ $a.a.s$.
\end{proposition}
\begin{proposition}
If $k = 1100$, then \emph{a.a.s.}\ $\tau_{k} \le 1.20365$ and $X(\tau_{k}) \ge (1-10^{-6}) n$.
\end{proposition}
After $k=1100$ phases, we have constructed a matching $M_{\tau_k}$ which \emph{a.a.s.}\ has at least $(1-10^{-6})n/2$ edges. 
To complete the perfect matching, we apply the 2-stage randomised algorithm from Section~\ref{warm-up-upper-bound}. First, we uncolour all vertices and edges of the graph. Then, we keep exactly $(1-10^{-6})n/2$ edges in $M_{\tau_k}$, and ``artificially'' delete the other edges in the matching, and put the incident vertices into $U_{\tau_k}$. Starting again from $t=0$, the argument
from Section \ref{warm-up-upper-bound} can be generalized to show that $X(t)$ and $R(t)$ follow the solutions of the differential equations~(\ref{x'}) and~(\ref{y'}) with initial conditions $x(0)=1-10^{-6}$ and $r(0)=0$. This is because
we can reproduce the analysis of the algorithm from Section \ref{warm-up-upper-bound} with $(1 - 10^{-6})n$ initial edges in the matching, opposed to none. Numerical solutions show that $x(0.00158)>1-\eps$, where $\eps=10^{-14}$. The clean-up
algorithm from Section ~\ref{warm-up-upper-bound} completes the construction in $100\sqrt{\eps} n =10^{-5} n$ rounds. This yields an upper bound of $1.20365+0.00158+10^{-5}=1.20524$ on $\tau_{\texttt{PM}}$, as in Theorem \ref{thm:main_upper_bound}.\qed


\section{Lower Bound}
Throughout the section, we will repeatedly apply the following claim when deriving bounds on $\tau_{\scr{P}}$
for a property $\scr{P}$:
\begin{proposition} \label{prop:equivalence}
Given $\alpha > 0$, suppose that for any strategy $\scr{S}$ and $0 < \delta < \alpha$, 
\[
    \mb{P}[G^{\scr{S}}_{T}(n) \in \scr{P}] \rightarrow 0
\]
as $n \rightarrow \infty$, where $T=T(n) = (\alpha-\delta) n$.
Then $\tau_{\scr{P}} \ge \alpha$.
\end{proposition}
\subsection{Warming-up}\label{sec:warm-up}

We start with the proof of Proposition~\ref{prop:warm_up_lowerbound} as a warm-up for our improved result. 
Suppose that we consider some strategy $\scr{S}$ of the player
which begins on the empty graph on vertex set $[n]$. Recall
that we say the vertex $j \in [n]$ is \textbf{covered} by the square $u_t$ arriving at round $t$,
provided $u_t = j$. The analogous definition extends to the circle $v_t$. 
Let $S_j(t)$ denote the number of squares covering $j$, and let $C_j(t)$ denote the number of circles covering $j$, after $t$ steps of the strategy. 

Suppose that $X(t)$ is the number of vertices of $G_t$
which are covered by at least one square. Let us fix
a constant $c < \ln(2)$, and consider $T=T(n) = cn$. 
Observe that if $G_t$ has a perfect matching
at time $t \ge 0$, then at least $n/2$ vertices of $[n]$ must have been covered
by at least one square; that is, $X(t) \ge n/2$. On the other hand,
$X(t) = n - \sum_{j \in [n]} \bm{1}_{\{S_j(t)=0\}}$,
and so 
\begin{eqnarray*}
    \mb{E}[X(t)] &=& \sum_{j\in[n]}(1 - \mb{P}[S_j(t)=0]) = n \left(1 -\left(1 - \frac{1}{n}\right)^{t} \right) \\
    &=& (1+o(1)) n (1 - \exp(-t/n)).
\end{eqnarray*}
Thus, combined with a second moment computation, one can conclude that for any fixed $t \ge 0$, \emph{a.a.s.}
\[
    X(t) = (1+o(1)) n (1 - \exp(-t/n)).
\]
Now, $T/n \le c$, and $c < \ln(2)$,
so \emph{a.a.s.}
\[
    X(T) \le n (1 + o(1)) (1 - \exp(-c)) < n/2.
\]
Therefore, \emph{a.a.s.}\ $G_T$ does not have a perfect matching. Since this holds for each $c < \ln(2)$,
Proposition \ref{prop:equivalence} implies the desired property and the proof is finished.

\subsection{Reducing to an Approximate Perfect Matching}


Given a constant $c > 0$ and a function $\omega = \omega(n)$, we say that a strategy $\scr{S}$ is $\omega$-\textbf{well-behaved}
for $cn$ rounds, if for all $0 \le t \le  \lfloor cn \rfloor$, each vertex of $G_{t}(n)$ is covered by at most $\omega$ circles.
In order to prove Theorem \ref{thm:main_lower_bound}, we work with strategies which are well-behaved
for an appropriate choice of $\omega$. We will also define a new property ${\mathcal P}$ which satisfies $\tau_{{\mathcal P}}\le \tau_{\texttt{PM}}$. Moreover, we will show that general strategies  do \textit{not} perform significantly better than well-behaved strategies for building a graph in ${\mathcal P}$. This allows us to restrict to analysing well-behaved strategies, and the various random variables yielded by these strategies. Since the strategies are well-behaved, the set of random variables under concern have the nice ``bounded-Lipschitz'' property that is essential for the application of the differential equation method. We start with the definition of the new property ${\mathcal P}$ on the family of directed graphs, and we use $D[S]$ to denote the directed graph induced by set $S \subseteq V(D)$ for a directed graph~$D$. 

\begin{definition}\label{def:approximate_pm}
Suppose that $D$ is a directed graph with $m$ arcs on vertex set $[n]$.
We say that $D$ has an $(\mu, \delta)$-\textbf{approximate perfect matching}, denoted by $D\in\texttt{PM}(\mu,\delta)$, if there exists a subset $S \subseteq [n]$
such that:
\begin{enumerate}[label=(\subscript{A}{{\arabic*}})]
    \item $D[S]$ has a perfect matching. \label{eqn:perfect_matching_subgraph} 
    \item All the vertices of $S$ have in-degree at most $\mu$ in $D$. \label{eqn:limited_in_degree}
    \item $|S| \ge n - 2 m \delta$. \label{eqn:large_subgraph}
\end{enumerate}
A perfect matching of $D[S]$ is then called a $(\mu,\delta)$-\textbf{approximate perfect matching} of~$D$.
\end{definition}
We first observe the following relation between $\texttt{PM}$ and $\texttt{PM}(\mu,\delta)$.

\begin{proposition}
For any function $\mu = \mu(n)$, it holds that
\[
    \tau_{\texttt{PM$(\mu, \mu^{-1})$}} \le \tau_{\texttt{PM}}.
\]

\end{proposition}

\begin{proof}
Suppose that $D$ is a directed graph which has a perfect matching $M$. Let ${\mathcal V}_1$ be the set of vertices whose in-degree is greater than $\mu(n)$.  Let ${\mathcal V}_2$ be the set of isolated vertices obtained once vertices in ${\mathcal V}_1$ are removed from $M$. Clearly, $|{\mathcal V}_2|\le |{\mathcal V}_1|\le m/\mu$. 
Let $S=[n]\setminus ({\mathcal V}_1\cup {\mathcal V}_2)$.  Obviously, $D[S]$ has a perfect matching, and $|S|\ge n-|{\mathcal V}_1|-|{\mathcal V}_2|\ge n-2|{\mathcal V}_1|\ge n-2m\mu^{-1}$. It follows that it is at least as easy to create an approximate perfect matching than to create a perfect one. The proof is completed. 
\end{proof}



Next, we confirm that general strategies do not behave significantly better than well-behaved strategies for constructing a directed graph in $\texttt{PM$(\mu, \mu^{-1})$}$, where we view every edge in $G_t$ as an arc from the square to the circle. Moreover we may safely restrict to strategies that construct only simple graphs, as intuitively relocating $v_t$ when $u_tv_t$ is already an edge would perform no worse, if not better. In the next proposition, we obtain a strategy $\scr{S}_2$ from a given strategy $\scr{S}_1$ by copying the moves of $\scr{S}_1$ whenever the move does not create a multiple edge, or create a vertex with too many circles. Otherwise, $\scr{S}_2$ relocates the circle to a different vertex (for counting purpose we colour relocated circles blue). We show that $\scr{S}_2$ performs at least as well as $\scr{S}_1$ quantitatively.

\begin{proposition} \label{prop:strategy_reduction}
For any constant $0 < c < 1$ and function $\mu =\mu(n)=\omega(1)$, 
if $\scr{S}_1$ is a strategy of the player, then there exists another strategy
$\scr{S}_2$ which is $2\mu$-well-behaved for $cn$ rounds, and satisfies
\begin{equation*} 
    \mb{P}[\text{$G_{cn}^{\scr{S}_1}(n)$ satisfies $\texttt{PM$(\mu, \mu^{-1})$}$}] \le \mb{P}[\text{$G_{cn}^{\scr{S}_2}(n)$ satisfies $\texttt{PM$(2\mu, \mu^{-1})$}$}] .
\end{equation*}
Moreover, $\scr{S}_2$ does not create any multi-edges, and so $G_{cn}^{\scr{S}_2}(n)$ is a simple graph.
\end{proposition}
\begin{proof}
In order to prove the claim,
we couple the execution of $\scr{S}_1$
with another strategy $\scr{S}_2$ which is
$2\mu$-well-behaved, and which constructs a directed graph in  
$\texttt{PM$(2\mu, \mu^{-1})$}$ no later than
when $\scr{S}_1$ 
constructs one in $\texttt{PM$(\mu, \mu^{-1})$}$.
We define $\scr{S}_2$ by \textit{stealing}
the strategy of $\scr{S}_1$, while making a slight
modification to avoid multi-edges and to ensure that no vertex of $[n]$ is covered
by more than $2\mu$ circles.

Suppose that the squares presented to strategy $\scr{S}_1$
are $u_{1}, \ldots , u_{cn}$, and whose corresponding
circles are $v_{1}^{1},\ldots ,v^{1}_{cn}$. We denote
the graph formed by $\scr{S}_1$ after $0 \le t \le cn$ rounds
by $G_{t}^{1}$. Let us define another strategy $\scr{S}_2$ with the same set of squares as presented to $\scr{S}_1$, whose
circles we denote by $v_{1}^{2},\ldots ,v_{cn}^{2}$.
 For each
$1 \le t \le cn$, suppose that $u_t$ is presented to $\scr{S}_2$, and
$G_{t-1}^{2}$ is the current graph constructed by $\scr{S}_2$ (where $G_0^2$
is the empty graph):
\begin{enumerate}
    \item If adding $u_{t}v_{t}^{1}$ to $G_{t-1}^{1}$ 
    does \textit{not} cause the number of
    circles on $v_{t}^{1}$ to exceed
    $\mu$, and $u_tv_t^{1}$ is \textit{not} an edge in $G_{t-1}^{2}$ (note here that we consider $G_{t-1}^{2}$ as an undirected graph),
    then set $v_{t}^{2}:= v_{t}^{1}$ \label{eqn:easy_round}.
    \item Otherwise, choose $v_{t}^{2} \in [n] \setminus N_{G_{t-1}^{2}}(u_t)$
    amongst those vertices which have at most $c/(1-c)$ blue circles, and then
    colour the circle $v_{t}^2$ blue.\label{eqn:problematic_round}
    \item Add $(u_t, v_{t}^{2})$ to $G_{t-1}^{2}$ to get $G_{t}^2$.
\end{enumerate}
We first observe that the strategy $\scr{S}_2$ is well-defined in that there always
exists a choice of $v_{t}^{2}$ which satisfies the requirements of $\eqref{eqn:problematic_round}$.
To see this, notice that trivially $|N_{G_{t-1}^{2}}(u_t)| \le c n$, and so there are at least
$(1 - c) n> 0$ vertices which are not adjacent to $u_t$ in $G_{t-1}^{2}$, as $c <1$ by assumption. Thus, since there
are at most $cn$ blue circles, a simple averaging argument ensures that there exists a vertex
of $[n] \setminus N_{G_{t-1}^{2}}(u_t)$ which is covered by at most $c/(1-c)$ blue circles.

We now argue that the strategy $\scr{S}_2$ satisfies all the desired properties of Proposition \ref{prop:strategy_reduction}. It is clear
that $\scr{S}_2$ does not create any multi-edges by definition, so it suffices to verify that
each vertex of $G_{cn}^{2}$ is covered by at most $2\mu$ circles, and that $G_{cn}^{2}$ is
in \texttt{PM$(2\mu, \mu^{-1})$} if $G_{cn}^{1}$ is in \texttt{PM$(\mu, \mu^{-1})$} . Observe that if we fix $j \in [n]$, then step \eqref{eqn:easy_round} places at 
most $\mu$ (uncoloured) circles on $j$, and step \eqref{eqn:problematic_round} places at most
$c/(1-c) +1<\mu$ blue vertices of $j$. Thus, in total there are at most $2\mu$ circles are placed on $j$,
and so the strategy $\scr{S}_2$ is $2\mu$-well-behaved. As a result,
if $G_{cn}^{1}$ has a $(\mu,\mu^{-1})$-approximate perfect matching, then $G_{cn}^{2}$
must have a $(2\mu, \mu^{-1})$-approximate perfect matching.
Thus, 
\[
\mb{P}[\text{$G_{cn}^{1}$ satisfies $\texttt{PM$(\mu, \mu^{-1})$}$}] \le \mb{P}[\text{$G_{cn}^{2}$ satisfies $\texttt{PM$(2\mu, \mu^{-1})$}]$}.
\]
and so the proof is complete.
\end{proof}

By combining the above propositions with Proposition \ref{prop:equivalence}, we get the following
lemma:

\begin{lemma}\label{lem:light_matching}
Given $0 < \alpha \le 1$ and $\mu = \mu(n) = \omega(1)$, suppose that for each $c < \alpha$ and each strategy $\scr{S}$
which is $2\mu$-well behaved for $cn$ rounds, it holds that
\[
    \mb{P}[\text{$G_{cn}^{\scr{S}}(n)$ satisfies $\texttt{PM$(2\mu, \mu^{-1})$}$}]  \rightarrow 0, \quad \mbox{as $n \rightarrow \infty$.}
\]
Then, $\tau_{\texttt{PM}} \ge \alpha$.

\end{lemma}

\subsection{Proving Theorem \ref{thm:main_lower_bound}}

In this section, we complete the proof of Theorem \ref{thm:main_lower_bound}.
Recall that $g: [0, \infty) \rightarrow \mb{R}$ is defined such that for $b \ge 0$,
\begin{equation} \label{eqn:first_bound}
    g(b) := 1 + \frac {1-2b}{2} \exp(-b) - (b+1) \exp(-2b) - \frac {1}{2} \exp(-3b),
\end{equation}
and in viewing $g$ as monotonely increasing on $[0, \infty)$, we have that
\[
\alpha =\min\{\alpha \ge 0: g(b)\ge 1/2 \}.
\]
Let $\mu=\sqrt{n}$ (indeed any function of $n$ that grows with $n$ in a sub-linear rate would work) and let $0<c<\alpha \le 1$ be an arbitrary constant. Define $\omega$ to be $2\mu$. This specifies the parameters for the set of $\omega$-well behaved strategies for $cn$ rounds that are considered below.

Suppose that $\scr{S}$ is an $\omega$-well behaved strategy
for $cn$ rounds. Since $c < \alpha \le 1$, we may apply
Lemma \ref{lem:light_matching} in order to prove Theorem \ref{thm:main_lower_bound}.
Thus, it suffices to show that
\begin{equation}\label{eqn:well_behaved_strategy}
    \mb{P}[\text{$G_{cn}^{\scr{S}}(n)$ satisfies $\texttt{PM$(\omega, \mu^{-1})$}$}]  \rightarrow 0,\quad\mbox{as $n \rightarrow \infty$}.
\end{equation}

For simplicity and without confusion, we remove $\scr{S}$ from the superscript.
In order to prove \eqref{eqn:well_behaved_strategy}, we prove that a certain set of  squares cannot be simultaneously contained in any $(\omega, \mu^{-1})$-approximate perfect matching of $G_{cn}$ and there are many such squares to exclude.

Given a vertex $j \in [n]$, we say that $j$ is \textbf{redundant} at time $t \ge 0$, provided the following conditions hold:
\begin{enumerate}[label=(\subscript{B}{{\arabic*}})]
    \item Vertex $j$ is covered by precisely one square, say $u_s$ for $s \le t$.
    \item The circle $v_{s}$ connected to $u_s$ by the player is covered by at least one
    square, which arrives after round $s$. \label{eqn:strategy_independece}
\end{enumerate}

We denote by $U(t)$ the number of redundant vertices at time $t$. Observe that \ref{eqn:strategy_independece} ensures
that $U(t)$ depends only on the placement of the squares, not the strategy $\scr{S}$. Observe also that the redundant square on $j$ and any square on $v_s$ cannot be contained simultaneously in any $(\omega, \mu^{-1})$-approximate perfect matching. Recall that $X(t)$ denotes the number of vertices covered by at least one square after step $t$.

Suppose that $j$ is redundant at time $t$ thanks to the arrival of the square $u_s$ at time $s \le t$.
We then say that $j$ is \textbf{well-positioned}, 
provided the vertex $v_s$ is also redundant at time $t$. Denote $W(t)$
as the number of well-positioned redundant vertices at time $t$. Clearly,
$W(t) \le U(t)$ by definition. We observe then the following inequality:

\begin{proposition} \label{prop:first_improvement}
If  $G_{t}$ has an $(\omega, \mu^{-1})$-approximate perfect matching at time $t \ge 0$,
then
\begin{equation}
    X(t) - U(t) + W(t) \ge \frac{n}{2} - \frac{3 t}{\mu}. \label{lowerbound1}
\end{equation}
\end{proposition}
\begin{proof}
Since $G_t$ has an $(\omega, \mu^{-1})$-approximate perfect matching, there exists some subset
$S \subseteq [n]$, such that $|S| \ge n - 2t /\mu$, for which
$G_{t}[S]$ has a perfect matching. Now, if $\scr{M}$ is a perfect matching
of $G_{t}[S]$, then clearly $|\scr{M}| = |S|/2 \ge n/2 - t/\mu$.

On the other hand, suppose $\scr{X}_{S}$ is the collection of vertices within
$S$ covered by at least one square, $\scr{U}_{S}$ is the collection  of
vertices of $S$ which are redundant, and $\scr{W}_{S}$
is the number of well-positioned vertices of $G_{t}[S]$.
Let us denote $X_{S}:= |\scr{X}_{S}|$, $U_{S}:= |\scr{U}_{S}|$ and $W_{S}:=|\scr{W}_S|$ for convenience.
We claim the following inequality:
\begin{equation} \label{eqn:subgraph_perfect_matching}
    X_{S} - U_{S} + W_{S}  \ge |\scr{M}| \ge n/2 - t/\mu.
\end{equation}
To see the left-hand side of this inequality, we first partition $S$
into $\scr{X}_{S} \setminus \scr{U}_{S}$, $\scr{U}_{S}$
and $\scr{Z}_{S}:= S \setminus \scr{X}_{S}$, thus ensuring that
\begin{equation} \label{eqn:perfect_matching_decomposition}
    |\scr{M}| = e_{\scr{M}}(\scr{X}_{S} \setminus \scr{U}_{S}, S) + e_{\scr{M}}(\scr{U}_S, \scr{U}_S) + e_{\scr{M}}(\scr{U}_S, \scr{Z}_{S}).
\end{equation}
Now, clearly $e_{\scr{M}}(\scr{X}_{S} \setminus \scr{U}_{S}, S) \le |\scr{X}_{S} \setminus \scr{U}_{S}|$.
On the other hand, we claim that $e_{\scr{M}}(\scr{U}_S, \scr{Z}_{S}) = 0$. This is because each vertex of $\scr{U}_S$
is redundant, and so its single out-neighbour is a vertex covered by at least one square, and thus 
cannot be a member of $\scr{Z}_S$. Similarly, its in-neighbours are covered by squares by definition,
and thus are also not in $\scr{Z}_S$.

It remains to upper bound $e_{\scr{M}}(\scr{U}_S, \scr{U}_S)$. In order
to do so, we partition $\scr{U}_S$
into $\scr{U}_{S} \setminus \scr{W}_{S}$ and $\scr{W}_S$, which ensures that
\[
    e_{\scr{M}}(\scr{U}_S, \scr{U}_S) = e_{\scr{M}}(\scr{U}_{S} \setminus \scr{W}_{S}, \scr{U}_{S} \setminus \scr{W}_{S}) + e_{\scr{M}}(\scr{W}_S,\scr{U}_S).
\]
Now, by definition each vertex $j \in \scr{U}_{S} \setminus \scr{W}_S$
has precisely one square on it, and the circle corresponding to this square lies in $\scr{X}_{S} \setminus \scr{U}_S$.
Thus, $e_{\scr{M}}(\scr{U}_{S} \setminus \scr{W}_S) = 0$. Moreover,
we can upper bound $e(\scr{W}_S, \scr{U}_S)$ by $|\scr{W}_S|$. 
By applying these bounds to \eqref{eqn:perfect_matching_decomposition},
\eqref{eqn:subgraph_perfect_matching} follows.

We can complete the proof by observing
that $X_{S} \le X(t)$, $W_{S} \le W(t)$, and $U_{S} \ge U(t) - (n -|S|) \ge U(t) - t/\mu$.
Thus, after rearranging \eqref{eqn:subgraph_perfect_matching}, we get that
\[
    X(t) - U(t) + W(t) \ge \frac{n}{2} -  \frac{3 t}{\mu},
\]
as claimed.
\end{proof}
As we saw in Proposition \ref{prop:warm_up_lowerbound}, it is easy to control the behaviour of
$X(t)$ for $t \ge 0$. Controlling the behaviour of $U(t)$ requires more care, and so
it is useful to define $Y(t)$ as the number of vertices of $G_t$ which are covered
by precisely one square. 
Let $H_{t}=(G_0,\ldots, G_t)$ be the history of the random graph process
up to step $t$. We then have that for each $t \ge 1$:
\begin{equation} \label{eqn:square_covered_vertices}
    \mb{E}[X(t+1) - X(t) \, | \, H_{t}]  = 1 - \frac{X(t)}{n},
\end{equation}
\begin{equation} \label{eqn:singularly_square_covered_vertices}
    \mb{E}[Y(t+1) - Y(t) \, | \, H_t] = \frac{n - X(t)}{n} - \frac{Y(t)}{n},
\end{equation}
and
\begin{equation} \label{eqn:redundant_vertices}
    \mb{E}[U(t+1) - U(t) \, |  \, H_{t}] = \frac{Y(t) - U(t)}{n} - \frac{U(t)}{n}.
\end{equation}
Verification of~(\ref{eqn:square_covered_vertices})--(\ref{eqn:redundant_vertices}) is straightforward. We briefly verify~(\ref{eqn:redundant_vertices}). The number of pairs of vertices $(u,v)$ such that $u$ is covered by exactly one square whose corresponding circle is on $v$ and $v$ has no squares arriving after that circle, is exactly $Y(t)-U(t)$ by step $t$. A redundant vertex is created if a square lands on $v$. On the other hand, a redundant vertex can be destroyed if it receives another square. This explains~(\ref{eqn:redundant_vertices}).

In order to control the behaviour of $W(t)$, we introduce another definition.
Let us say that a vertex $j_1$ is \textbf{dangerous} at time $t \ge 0$,
provided there exist times $t_{1} < t_{2} \le t$, such that the following hold:
\begin{enumerate}[label=(\subscript{C}{{\arabic*}})]
    \item At time $t_1$, the square $u_{t_1}$ lands on $j_1$, and connects its circle $v_{t_1}$ to some
    vertex $j_2 \neq j_1$.
    \item At time $t_2$, the vertex $j_2$ is covered by the square $u_{t_2}$, at which point
    the circle $v_{j_2}$ is placed on some vertex $j_{3} \notin \{j_1,j_{2}\}$.
    \item No squares land on $j_3$ between times $t_2$ and $t$, and
    $u_{t_1}$ and $u_{t_2}$ are the only squares on $j_1$ and $j_2$, respectively.
\end{enumerate}
Let us  denote $D(t)$ as the number of dangerous vertices at time $t$. Observe then
that
\begin{equation} \label{eqn:dangerous_vertices}
    \mb{E}[ D(t+1) - D(t) \, | \, H_{t}] = \frac{Y(t) - U(t)}{n} - \frac{3 D(t)}{n},
\end{equation}
where we recall that $Y(t)$ is the number of vertices covered by exactly one square,
and $U(t)$ is the number of redundant vertices. Note that the equality in \eqref{eqn:dangerous_vertices}
follows since the strategy we are analyzing does \textit{not} create multi-edges by assumption.
We can use $D(t)$ to help us track the behaviour of $W(t)$ at time $t$;
that is, the well-positioned vertices. Specifically, observe that
\begin{equation*} \label{eqn:well_positioned_vertices}
    \mb{E}[ W(t+1) - W(t) \, | \, H_t] =\frac{D(t)}{n} - \frac{2 W(t)}{n}.
\end{equation*}

Now, by assumption $\scr{S}$ is $\omega$-well behaved---that is,
there are at most $\omega = \omega(n)$ circles on any vertex of $G_t$. We can thus bound 
the worst case one-step changes of our random 
variables by $\omega(n)$. More explicitly, observe that for each $t \ge 0$ 
\begin{equation*}
    |X(t+1) - X(t)| \le \omega(n),
\end{equation*}
and the analogous upper bound is also true for the other random variables.

Suppose now that the functions $x,y,u,d,w: [0, \infty) \rightarrow \mb{R}$ are the unique solution to
the following system of differential equations:
\begin{eqnarray*}
    x' &=& 1 - x, \\
    y' &=& 1 - x - y, \\
    u' &=& y - 2 u, \\
    d' &=& y - u - 3 d,\\
    w' &=& d -2 w,
\end{eqnarray*}
with initial conditions $x(0) = y(0)= u(0)= d(0) =w(0)= 0$.
These functions have the following closed form solution:
\begin{itemize}
    \item $x(b) = 1 - \exp(-b)$,
    \item $y(b) = b \exp(-b)$,
    \item $u(b) = (b-1) \exp(-b) + \exp(-2b)$.
    \item $d(b) = \frac{1}{2} \exp(-b) + \frac{1}{2} \exp(-3b) -\exp(-2b)$,
    \item $w(b) = \frac{1}{2} \exp(-b) - b \exp(-2b) - \frac{1}{2} \exp(-3b)$.
\end{itemize}

The following lemma follows immediately by Nick Wormald's differential equation method (see, for example,~\cite[Theorem 5.1]{de} or~\cite{warnke2019wormalds}).

\begin{lemma} \label{lem:differential_equation_application} There exists a function  $\eps=\eps(n) = o(1)$
such that \emph{a.a.s.}, for all $0 \le t \le cn$,
\begin{equation*}
    \max\{|X(t) -x(t/n) n|,|U(t) - u(t/n) n| , |W(t) - w(t/n) n|\} \le \eps n.
\end{equation*}
\end{lemma}

By applying Lemma \ref{lem:differential_equation_application},
we are guaranteed the existence of a function $\eps = \eps(n) = o(1)$,
such that \emph{a.a.s.}\ for all $t \ge 0$,
\begin{itemize}
    \item $X(t)/n \le 1 - \exp(-t/n) + \eps$,
    \item $U(t)/n \ge \left( \frac{t}{n} -1\right) \exp(-t/n) + \exp(-2t/n) - \eps$,
    \item $W(t)/n \le \left( \frac{1}{2} \exp(-t/n) - \frac{1}{2} \exp(-3t/n) - \frac{t}{n} \exp(-2 t/n) \right) + \eps$.
\end{itemize}

Let us now assume that $\alpha \ge 0$ is the infimum over
those $b \ge 0$ such that
$g(b) \ge 1/2$,
where $g(b)$ is defined in~(\ref{eqn:first_bound}).
Now, $c < \alpha$ and
$T(n):= \lfloor c n \rfloor$ satisfies $T = (1-o(1))c n$. Thus,
\[
    \frac{X(T)}{n} + \frac{W(T)}{n} - \frac{U(T)}{n} \le  g(c) + 3 \eps.
\]
Moreover, $g(c) <1/2$ as $c < \alpha$, and so since $\eps,\mu^{-1} = o(1)$,
it follows that \emph{a.a.s.}\ 
\[
    X(T) + W(T) - U(T) < n/2 - 3T/\mu.
\]
Thus, $G_T$ does not have an $(\omega,\mu^{-1})$-approximate perfect matching \emph{a.a.s.}\ by Proposition \ref{prop:first_improvement}. As $c < \alpha$ was arbitrary, the proof of Theorem \ref{thm:main_lower_bound} is complete after applying Lemma \ref{lem:light_matching}. \qed

\section{Conclusion}

We have reduced the gap between the previous best upper and lower bounds on $\tau_{{\tt PM}}$ by roughly a factor of four. That being said,
we do not believe that any of our new bounds are tight. For instance, in the case of our upper bound, our strategy does not
make use of subsequent squares which land on a green vertex. One way to improve the algorithm's performance would be to control these
additional squares. The other way is to consider longer augmenting paths. However, there are not many saturated vertices with more than one square and thus the numerical improvement is not so significant. Tracing longer augmenting paths is harder to analyse and we did not attempt this. However, algorithmic simulations\footnote{We did not simulate a longer augmenting path algorithm ourselves, however we thank an anonymous reviewer for informing us that their simulation suggests an upper bound of $1.09$.} suggest that this approach will lead to an improved upper bound. In the case of our lower bound, we have identified certain subgraphs which  preclude the strategy from building a perfect matching too quickly, and which no strategy can avoid creating. To improve our lower bound,
one could track more complicated subgraphs which also prevent a perfect matching from being created too quickly. For instance, one could track
the number of vertices which are covered by two squares, and whose neighbours are themselves \textit{not} redundant. However, such subgraphs
can be avoided by the strategy, and so they are harder to control. More importantly, these subgraphs are much rarer
than redundant vertices, and so their contribution to the improvement of the lower bound is rather small.

We believe it is an interesting open question to resolve the remaining gap between our bounds on
$\tau_{{\tt PM}}$. More generally, we believe it is an interesting direction to better understand the value of $\tau_{{\tt H}}$,
when $H$ is a graph with absolutely bounded maximum degree $\Delta$ which spans all the vertices in $[n]$ . For instance, improving
upon the upper bound of $3/2( \Delta + o(\Delta))$ of~\cite{beneliezer2020fast} remains an intriguing open question
when $\Delta$ is small.







\bibliographystyle{plain}

\bibliography{refs.bib}


\end{document}